\documentclass[reqno]{amsart}
\usepackage{latexsym,amsmath,amssymb
}
\usepackage{amsthm}
\usepackage{enumerate}
\usepackage[shortlabels]{enumitem}
\usepackage[
colorlinks=true,linkcolor=blue,citecolor=red]{hyperref} 
\usepackage[T1]{fontenc}

\makeatletter
\@addtoreset{figure}{section}
\def\thefigure{\thesection.\@arabic\c@figure}
\def\fps@figure{h,t}
\@addtoreset{table}{bsection}

\def\thetable{\thesection.\@arabic\c@table}
\def\fps@table{h, t}
\@addtoreset{equation}{section}

\makeatother

\DeclareMathAlphabet\mathbfcal{OMS}{cmsy}{b}{n}

\newtheorem{theorem}{Theorem}
\newtheorem{corollary}[theorem]{Corollary}
\newtheorem{definition}[theorem]{Definition}

\newtheorem{lemma}[theorem]{Lemma}

\newtheorem{proposition}[theorem]{Proposition}
\newtheorem{remark}[theorem]{Remark}

\numberwithin{theorem}{section}
\numberwithin{equation}{section}

\newcommand{\1}{{\bf 1}}

\newcommand{\Ci}{\mathcal{C}^\infty}

\newcommand{\ee}{{\rm e}}

\newcommand{\id}{{\rm id}}
\newcommand{\ie}{{\rm i}}

\newcommand{\dual}[2]{\langle #1, #2\rangle}

\newcommand{\CC}{{\mathbb C}}

\newcommand{\RR}{{\mathbb R}}

\newcommand{\Bc}{{\mathcal B}}
\newcommand{\Cc}{{\mathcal C}}

\newcommand{\Hc}{{\mathcal H}}

\newcommand{\Lc}{{\mathcal L}}
\newcommand{\Sc}{{\mathcal S}}

\newcommand{\Vc}{{\mathcal V}}

\newcommand{\Xc}{{\mathcal X}}
\newcommand{\Yc}{{\mathcal Y}}

\title[Distribution vectors in Lie group representations]{Invariant Hilbert spaces of distribution vectors in Lie group representations}

\author{Ingrid Belti\c t\u a}
\author{Daniel Belti\c t\u a}
\address{Institute of Mathematics ``Simion Stoilow'' of the Romanian Academy,
	P.O. Box 1-764, Bucharest, Romania}

\email{Ingrid.Beltita@imar.ro, ingrid.beltita@gmail.com}
\email{Daniel.Beltita@imar.ro, beltita@gmail.com}
\thanks{This work was supported by a grant of the Ministry of Research, Innovation and Digitization, CNCS--UEFISCDI, project number PN-IV-P1-PCE-2023-0264, within PNCDI IV}

\begin{document}
	
	\parskip=5pt

	\begin{abstract}
	For every unitary irreducible representation of a 
	Lie group we prove that the representation Hilbert space is the only nonzero invariant Hilbert space of distribution vectors. 
	\end{abstract}

	\maketitle

	\section{Introduction}
	
	Hilbert subspaces of topological vector spaces in the sense of Laurent Schwartz \cite{Sch64} provide an abstract version of reproducing kernels of function spaces and they thus play an important role in many areas of mathematical physics, analysis, and representation theory of Lie groups,  
	as shown for instance in \cite{Ne00}. 
	There has been a recent interest in uniqueness of Hilbert subspaces of spaces of tempered distributions that are invariant to modulations and translations, partly motivated by problems related to time-frequency analysis and modulation spaces, see \cite{ToGuMaRa21}, \cite{BaMoPi24}, \cite{RaToVi25}.

	In this paper we prove an abstract uniqueness theorem for Lie group representations, which recovers some of the preceding results in the special case  of  Schr\"odinger representations of the Heisenberg groups. 
	Specifically, we show that if $\pi\colon G\to\Bc(\Hc)$ is an irreducible unitary  representation of a  
	Lie group $G$, 
	then $\Hc$ is the only continuously embedded Hilbert space into the space of 
	distribution vectors $\Hc^{-\infty}$ which is nonzero and is invariant to the action of $G$ 
	(Theorem~\ref{unique}). 
	In section~\ref{nilpotent} we give a version of that result replacing 
	the action of the group 
	by  
	the action of the Schwartz algebra on a nilpotent Lie group. 
	 
	As explained in \cite{ToGuMaRa21}, the minimality properties of the Feichtinger algebra provided the initial motivation for uniqueness results on Hilbert spaces that are invariant to translations and modulations. 
	 In that spirit, one used the realization of the Hilbert space as $L^2(\RR^d)$, 
 which facilitated the techniques of distribution theory and Hermite functions.
	 Our approach is completely different, 
	 inasmuch as we rather rely on some basic results on representations of 
	 Lie groups  
	 and thus, in Corollary~\ref{RaToVi25_Thm0.3}, we provide a new perspective on  \cite[Thm. 0.3]{RaToVi25}.

	\section{Preliminaries}

	The framework of the present paper is provided by the notions of smooth vectors and distribution vectors with respect to unitary representations of Lie groups. 
	Details of this theory can be found in \cite{Ne10}, \cite{BeBe10}, \cite{BeBe11}, \cite{BeBeGa18}, 
	and the references therein. 
	
		For every locally convex, topological $\CC$-vector space $\Yc$, we denote by $\Lc(\Yc)$ the associative algebra of continuous linear operators on $\Yc$. 
		and by $\Yc'$ the space of continuous antilinear functionals on $\Yc$ endowed with its $w^*$-topology, i.e., the topology of pointwise convergence. 
		If $\langle\cdot,\cdot\rangle\colon\Yc'\times\Yc\to\CC$ is the corresponding sesquilinear duality pairing (antilinear in its second variable), 
		then for every $T\in\Lc(\Yc)$ there is a unique operator $T'\in\Lc(\Yc')$ satisfying 
		$\langle T'\eta,y\rangle=\langle \eta,Ty\rangle$ for all $\eta\in\Yc'$ and $y\in\Yc$. 
		If $\Xc$ is a Banach space, since the continuous linear operators on $\Xc$  are precisely the bounded operators, we denote $\Bc(\Xc):=\Lc(\Xc)$.
		
		We  need a version of Schur's lemma for unbounded intertwining operators. 
		
		\begin{definition}\normalfont 
		Let $G$ be a locally compact group and  $\pi\colon G \to \Bc(\Hc_j)$, $j=1, 2$, two unitary representations of the group $G$.
		An unbounded operator $S$ from $\Hc_1$ to $\Hc_2$ with domain 
		$D(S) \subseteq \Hc_1$ is an \emph{intertwining operator} for the representations $\pi_1$ and $\pi_2$ if 
		for every $g \in G$ and 
		$x\in D(S)$ we have $\pi(g)x\in D(S)$ and
		$$ S \pi_1(g) x  =\pi_2(g) Sx. $$
\end{definition}
 
 \begin{lemma}\label{Schur}
Let $G$ be a locally compact group and $\pi\colon G \to \Bc(\Hc_j)$, $j=1, 2$, two unitary representations of the group $G$ with $\pi_1$  irreducible. 
Assume that $S$ is non-zero closable intertwining operator for $\pi_1$ and $\pi_2$. 
Then the closure $\overline{S}$ of $S$ is a bounded operator 
from $\Hc_1$ to $\Hc_2$ and there is $\lambda>0$ such that $\lambda\overline{S}\colon \Hc_1\to \Hc_2$ is an isometry. 
\end{lemma}
\begin{proof}
This is  a direct consequence of \cite[Lemma~II.2.6 and Thm.~II.2.8]{Ne00}.
\end{proof}

In the present paper, unless otherwise mentioned, we  use the following notation.
	\begin{itemize}
		\item $G$ is a 
		Lie group;
		\item $\pi\colon G\to \Bc(\Hc)$ is a
		(continuous) unitary representation; 
		\item $\Hc^\infty:=\{v\in \Hc\mid \pi(\cdot)v\in\Ci(G,\Hc)\}$ is the Fr\'echet space of smooth vectors; 
		\item $\Hc^{-\infty}$ is the space of distribution vectors, 
		i.e., the space of continuous antilinear functionals on $\Hc^\infty$, endowed with the $w^*$-topology;
		\item $\langle\cdot,\cdot\rangle\colon\Hc^{-\infty}\times\Hc^\infty\to\CC$ 
		is the corresponding sesquilinear duality pairing, antilinear in its second variable, 
		which agrees with the scalar product of the Hilbert space $\Hc$;  
		\item we have the group representations 
		$$\pi^\infty\colon G\to\Lc(\Hc^\infty)\text{ and }
		\pi^{-\infty}\colon G\to\Lc(\Hc^{-\infty})$$ 
		satisfying $\pi^\infty(g)=\pi(g)\vert_{\Hc^\infty}$ and $\pi^{-\infty}(g):=(\pi^\infty(g^{-1}))'$ for every $g\in G$;
		\item $C_0^\infty(G)$ is the space of compactly supported smooth functions on $G$. 
		\end{itemize}

\begin{remark}
\label{triple}
\normalfont 
In the setting above, 
 the natural continuous injective linear maps 
\begin{equation}
	\label{triple_eq1}
	\Hc^\infty\hookrightarrow \Hc\hookrightarrow \Hc^{-\infty}
\end{equation}
intertwine the group representations $\pi^\infty$, $\pi$, and $\pi^{-\infty}$. 

Using integration with respect to 
a 
Haar measure of $G$, 
the unitary representation~$\pi$ gives rise to a Banach $*$-algebra representation 
$\pi\colon L^1(G)\to\Bc(\Hc)$. 
Moreover,  we have 
$\pi(C_0^\infty(G))\Hc^\infty\subseteq\Hc^\infty$
with dual  map 
$$\pi^{-\infty}\colon C_0^\infty(G) \to\Lc(\Hc^{-\infty}),\quad 
\pi^{-\infty}(\varphi):=(\pi(\varphi^*))'.
$$
Then 
\begin{equation}
\label{triple_eq2}
\pi^{-\infty}(C_0^\infty (G))\Hc^{-\infty}=\Hc^\infty.
\end{equation}
by \cite[Thm.~3.1 and 3.2]{DM78}. 
 \end{remark}

\section{The main result} 

\begin{theorem}
\label{unique}
Let $\pi\colon G\to\Bc(\Hc)$ be an irreducible unitary representation of a 
Lie group $G$. 
Let $\Xc\subseteq\Hc^{-\infty}$ be a $\CC$-linear subspace endowed with the structure of a complex Hilbert space satisfying the following hypotheses: 
\begin{enumerate}
[start=1,label={\rm(H\arabic*)}]
	\item\label{unique_item1} 
	the inclusion map $\Xc\hookrightarrow\Hc^{-\infty}$ is continuous; 
	\item\label{unique_item2}
	 $\pi^{-\infty}(G)\Xc\subseteq\Xc$; 
	\item\label{unique_item3} 
	for every $g\in G$ the operator $\pi_\Xc(g):=\pi^{-\infty}(g)\vert_\Xc\colon\Xc\to\Xc$ is an isometry.
\end{enumerate}
Then either $\Xc=\{0\}$ or $\Xc=\Hc$. 
\end{theorem}

For the proof of the above theorem we 
need the following lemma.

\begin{lemma}
	\label{cont}
	Let $\Gamma$ be a topological group, 
	$\Xc$ a complex Hilbert space, $\rho\colon \Gamma\to \Bc(\Xc)$ a mapping, 
	and for every $x\in\Xc$ denote 
	$\rho^x\colon \Gamma\to\Xc$, $\rho^x(g):=\rho(g)x$. 
	Assume that the following conditions are satisfied: 
	\begin{enumerate}[{\rm(a)}]
		\item For every $g,h\in \Gamma$ we have $\rho(gh)=\rho(g)\rho(h)$ and moreover $\rho(\1)=\id_\Xc$. 
		\item For every $g\in \Gamma$ the operator $\rho(g)\colon \Xc\to\Xc$ is an isometry. 
		\item There exists a subset $A\subseteq\Xc$ that spans a dense linear subspace of $\Xc$ and $\rho^x\colon \Gamma\to \Xc$ is continuous for every $x\in A$. 
	\end{enumerate}
	Then $\rho$ is a continuous unitary representation  of the topological group $\Gamma$. 
\end{lemma}

\begin{proof}
	For every $g\in \Gamma$ we have $\id_\Xc=\rho(\1)=\rho(g)\rho(g^{-1})$, 
	hence $\rho(g)\in\Bc(\Xc)$ is a surjective isometry.
	that is, a unitary operator. 
	It remains to prove that $\rho^x\colon \Gamma\to \Xc$ is continuous for every $x\in \Xc$. 
	(See e.g., \cite[Lemma 5.2]{Ne10}.)
	To this end we consider the set $\Xc_0:=\{x\in\Xc\mid \rho^x\in\Cc(\Gamma,\Xc)\}$. 
	It is clear that $\Xc_0$ is a linear subspace of $\Xc$ and $A\subseteq\Xc_0$, hence $\Xc_0$ is dense in $\Xc$. 
	On the other hand, for arbitrary $x,y\in\Xc$ and $g\in \Gamma$ we have $\Vert \rho^x(g)-\rho^y(g)\Vert =\Vert \rho(g)x-\rho(g)y\Vert=\Vert x-y\Vert$. 
	This directly implies that $\Xc_0$ is closed in $\Xc$, hence $\Xc_0=\Xc$. 
\end{proof}

\begin{proof}[Proof of Theorem~\ref{unique}]
We assume $\Xc\ne\{0\}$. 
For every $v\in\Hc^\infty$ we have $\overline{\langle\cdot,v\rangle}\in(\Hc^{-\infty})'$, 
hence $\overline{\langle\cdot,v\rangle}\vert_\Xc\in\Xc'$ by \ref{unique_item1}. 
Since $\Xc$ is a complex Hilbert space, there exists a uniquely determined vector $\iota(v)\in\Xc$ satisfying 
\begin{equation}
\label{unique_proof_eq1}
\langle\cdot,v\rangle\vert_\Xc=\langle\cdot,\iota(v)\rangle_\Xc,
\end{equation} 
where $\langle\cdot,\cdot\rangle_\Xc$ denotes the scalar product of $\Xc$. 
We thus obtain a linear operator 
$$\iota\colon \Hc^\infty\to\Xc,$$
such that 
\begin{equation*} 
\iota(v) = v \quad \text{for } v \in \Hc^\infty \cap \Xc.
\end{equation*}
The equality \eqref{unique_proof_eq1} shows that $\iota$ is the adjoint of the inclusion map $\Xc\hookrightarrow\Hc^{-\infty}$. 
The range of the operator $\iota$ is dense in $\Xc$ since the above inclusion map is injective.
In fact, if $x_0\in\Xc$ satisfies $x_0\perp \iota(\Hc^\infty)$ in $\Xc$ 
then, by \eqref{unique_proof_eq1}, 
we have $\langle x_0,v\rangle=0$ for every $v\in\Hc^\infty$, that is, $x_0=0$.

In order to prove that $\iota$ is a continuous operator, one can use the closed graph theorem since $\Hc^\infty$ is a Fr\'echet space and $\Xc$ is a Hilbert space. 
Specifically, if $\lim\limits_{n\to\infty}v_n=v$ in $\Hc^\infty$ and $\lim\limits_{n\to\infty}\iota(v_n)=x$ in $\Xc$, 
then $\iota(v)=x$ by \eqref{unique_proof_eq1}. 

For every $g\in G$, $v\in\Hc^\infty$ and $x\in\Xc$ we have 
\begin{align*}
\langle x,\iota(\pi^\infty(g)v)\rangle_\Xc
&=\langle x,\pi^\infty(g)v\rangle
=\langle \pi^{-\infty}(g^{-1})x, v\rangle
=\langle \pi_\Xc(g^{-1})x, v\rangle \\
& =\langle \pi_\Xc(g^{-1})x,\iota(v)\rangle_\Xc 
=\langle x,\pi_\Xc(g)\iota(v)\rangle_\Xc, 
\end{align*}
hence 
\begin{equation}
	\label{unique_proof_eq2}
(\forall g\in G, v\in\Hc^\infty)\quad 	\iota(\pi^\infty(g)v)=\pi_\Xc(g)\iota(v).
\end{equation} 
Since $\iota\colon\Hc^\infty\to\Xc$ is a continuous linear operator
and the mapping $G\ni g\mapsto\pi^\infty(g)v\in\Hc^\infty$ is smooth,  
as $\pi\colon G\to\Bc(\Hc)$ is a continuous representation 
(cf. \cite[Prop. 1.2]{Po72}), 
it also follows by \eqref{unique_proof_eq2} that, for every $x\in\iota(\Hc^\infty)\subseteq\Xc$, the mapping $\rho_\Xc^x\colon G\to\Xc$, $\pi_\Xc^x(g):=\pi_\Xc(g)x$, is smooth and in particular continuous. 
We proved above that $\iota(\Hc^\infty)$ is dense in $\Xc$, 
hence by~\ref{unique_item3}, Lemma~\ref{cont} is applicable and 
ensures that $\pi_\Xc\colon G\to\Bc(\Xc)$ is a continuous unitary representation.
Moreover,~\eqref{unique_proof_eq2} shows that  $\iota(\Hc^\infty)\subseteq \Xc^\infty$, 
hence
\begin{equation}\label{unique_proof_eq2.1}
\Hc^\infty \cap \Xc \subseteq \Xc^\infty.
\end{equation}
 On the other hand, we recall from \eqref{triple_eq2} that  
$\pi^{-\infty}(C_0^\infty(G)))\Hc^{-\infty}=\Hc^\infty$;
in particular, 
$\pi_{\Xc}(C_0^\infty(G))\Xc\subseteq\Hc^\infty$.
Then Dixmier-Malliavin's theorem \cite[Thm.~3.2]{DM78} for the continuous representation 
$\pi_\Xc$ shows that $\Xc^\infty\subseteq \Hc^\infty$ which, along with 
\eqref{unique_proof_eq2.1} gives
\begin{equation}\label{unique_proof_eq2.2}
\Xc^\infty=\Hc^\infty\cap \Xc.
\end{equation} 
The subspace $\Xc^\infty$ is then contained in $\Hc$ and dense in $\Xc$ . 
Moreover
since it is invariant under $\pi(g)$, $g\in G$, and the representation $\pi\colon G \to \Bc(\Hc)$ is irreducible, it follows that $\Xc^\infty$  is dense in $\Hc$, as well.

We  
now regard 
the operator $\iota$ above as an unbounded operator defined on $\Hc$ with domain $\Hc^\infty$. 
A vector $x\in \Xc$ belongs to the domain of the adjoint $\iota^*$ of $\iota$ if and only if the map 
$v\mapsto \langle x, \iota(v)\rangle_\Xc=\langle x,v\rangle
 $
is continuous in the norm of $\Hc$. 
Since 
the duality pairing $\langle \cdot ,\cdot\rangle$
 agrees with the scalar product of $\Hc$, we obtain that  
$\Xc^\infty$ is contained in the domain of $\iota^*$. 
It follows that $\iota^*$ is densely defined, therefore $\iota$ is a closable operator. 
Using now \eqref{unique_proof_eq2} and Lemma~\ref{Schur}
we get that there 
are 
an isometry $\iota_0\colon \Hc \to \Xc$ and a constant $\lambda_0>0$ 
such that $\iota = \lambda_0 \iota_0\vert_{\Hc^\infty}$. 
As the range of $\iota$, hence of $\iota_0$, is dense in $\Xc$, 
it follows that the operator $\iota_0\colon \Hc\to\Xc$ is unitary. 
Moreover, since $\iota\vert_{\Xc^\infty}$ is the identity map, and 
$\Xc^\infty$ is dense both in $\Xc$ and $\Hc$, we get that $\Hc=\Xc$ as sets and 
the scalar products of these Hilbert spaces are a suitable scalar multiple of each other. 
\end{proof}

\begin{remark}\label{means}
\normalfont 
In the setting of Theorem~\ref{unique}, we denote by $G_d$ the group $G$ endowed with its discrete topology.  
If we assume that $G_d$ is amenable
and  \ref{unique_item2} holds true, 
then \ref{unique_item3} is equivalent to the following, seemingly weaker, condition: 
\begin{enumerate}
	[start=4,label={\rm(H\arabic*)}]
	\item\label{unique_item4} 
	there exist a constant $C>0$ and an equivalent Hilbert space norm $\Vert\cdot\Vert'_\Xc$ such that for every $g\in G$ and $x\in\Xc$ we have $\Vert\pi_\Xc(g)x\Vert'_\Xc\le C\Vert x\Vert'_\Xc$.  
\end{enumerate}
In fact, clearly \ref{unique_item3}$\Rightarrow$\ref{unique_item4} with $\Vert\cdot\Vert'_\Xc=\Vert\cdot\Vert_\Xc$ and $C=1$. 
For the converse implication, 
let us assume \ref{unique_item4}. 
Since $G_d$ is amenable, by 
 \cite[Thm. 3.4.1]{Gr69} and its proof, one can define an equivalent Hilbert space norm $\Vert\cdot\Vert_\Xc$ on $\Xc$ for which the representation $\pi_\Xc\colon G_d\to\Bc(\Xc)$ is unitary, that is, \ref{unique_item3} holds true. 

We note that if the Lie group $G$ is connected, the condition that $G_d$ is amenable is equivalent with $G$ being solvable.  (See \cite[Thm.~3.9]{Pa88}.)

In the special case when $G$ is a Heisenberg group and $\pi$ is a Schr\"odinger representation as in the proof of Corollary~\ref{RaToVi25_Thm0.3} below, the above remark provides a new approach to the result of \cite[Lemma 1.3]{RaToVi25}.
\end{remark}

We are now in a position to provide a new proof of \cite[Thm. 0.3]{RaToVi25} as follows. 

\begin{corollary}
\label{RaToVi25_Thm0.3}
Let $\Xc\hookrightarrow\Sc'(\RR^d)$ be a continuous inclusion of a nonzero complex Hilbert space into the space of tempered distributions on $\RR^d$, satisfying the following condition: 
 there exists a constant $C_0>0$ such that for all $f\in\Xc$ and $x,\xi\in\RR^d$ we have $f(\cdot-x)\ee^{\ie\langle \cdot,\xi\rangle}\in\Xc$ and $\Vert f(\cdot-x)\ee^{\ie\langle \cdot,\xi\rangle}\Vert_\Xc\le C_0\Vert f\Vert_\Xc$. 

Then $\Xc=L^2(\RR^d)$ as vector spaces, and the norms $\Vert\cdot\Vert_\Xc$ and $\Vert\cdot\Vert_{L^2(\RR^d)}$ are equivalent. 
\end{corollary}

\begin{proof}
We consider  
the Heisenberg group $G=H_d=\RR^d\times\RR^d\times\RR$ 
with its group operation 
$$(x,\xi,t)\cdot(y,\eta,s):=(x+y,\xi+\eta,t+s+(\langle\xi,y\rangle-\langle\eta,x\rangle)/2)$$
and the Schr\"odinger representation 
$\pi\colon H_d\to \Bc(L^2(\RR^d))$ given by 
$$(\pi(x,\xi,t)f)(y)=\ee^{\ie(\langle\xi,y\rangle+t+\langle\xi,x\rangle/2)}f(y+x).$$
(See e.g., \cite[Ex. 2.2.6]{CG90}.)
Then we have $\Hc^\infty=\Sc(\RR^d)$ hence $\Hc^{-\infty}=\Sc'(\RR^d)$. 
Thus the assumption on $\Xc$ implies that the hypotheses \ref{unique_item1} and \ref{unique_item2} in Theorem~\ref{unique}, as well as \ref{unique_item4} in Remark~\ref{means} are satisfied. 
Now the assertion follows by Theorem~\ref{unique} along with Remark~\ref{means}. 
\end{proof}

\section{A result for nilpotent Lie groups}\label{nilpotent}

We now assume that $G$ is a 1-connected nilpotent Lie group
and $\Sc(G)$ is the associative Fr\'echet algebra of Schwartz functions on $G$ with the convolution product, hence $\Sc(G)$ is a $*$-subalgebra of $L^1(G)$. 
Let  $\pi\colon G\to \Bc(\Hc)$ be an irreducible (continuous) unitary representation. 
Using the Banach $*$-algebra representation 
$\pi\colon L^1(G)\to\Bc(\Hc)$, 
 we have 
$\pi(\Sc(G))\Hc^\infty\subseteq\Hc^\infty$, 
and this gives rise to the algebra representation 
$$\pi^\infty\colon \Sc(G)\to\Lc(\Hc^\infty).$$
Also, the map 
$$\pi^{-\infty}\colon \Sc(G)\to\Lc(\Hc^{-\infty}),\quad 
\pi^{-\infty}(\varphi):=(\pi(\varphi^*))'
$$
is an algebra representation.  

The maps in \eqref{triple_eq1} 
intertwine the representations $\pi^\infty$, $\pi\vert_{\Sc(G)}$, $\pi^{-\infty}$ of the associative $*$-algebra $\Sc(G)$ 
and $\pi^{-\infty}(\Sc(G))\Hc^{-\infty}\subseteq\Hc^\infty$. 
By \eqref{triple_eq2}  we  then have 
\begin{equation}
\label{triple_eq2bis}
\pi^{-\infty}(\Sc(G))\Hc^{-\infty}=\Hc^\infty. 
\end{equation}

\begin{lemma}
\label{alg_irred}
If $\Vc\subseteq\Hc^\infty$ is a $\CC$-linear subspace satisfying $\pi^\infty(\Sc(G))\Vc\subseteq\Vc$ then  either $\Vc=\{0\}$ or $\Vc=\Hc^\infty$, that is, the representation 
$\pi^\infty\colon \Sc(G)\to\Lc(\Hc^\infty)$ is algebraically irreducible.
\end{lemma}

\begin{proof}
See \cite[Cor. 3.4.1]{Ho77}, 
and also \cite[Cor. 3.4]{dC87} 
and \cite[Thm. 2.9]{BeBeGa18}.
\end{proof}

Before continuing the preparations for 
the proof of Theorem~\ref{unique_nilp}, 
we derive the following simple result that has a similar flavour and may hold an independent interest.

\begin{proposition}
	\label{w_irred} In the conditions above, 
	if $\Yc\subseteq \Hc^{-\infty}$ is a $w^*$-closed linear subspace with  $\pi^{-\infty}(\Sc(G))\Yc\subseteq\Yc$, then we have either $\Yc=\{0\}$ or $\Yc=\Hc^{-\infty}$. 
\end{proposition}

\begin{proof}
	Let $\Yc^\perp:=\{v\in\Hc^\infty\mid\langle\Yc,v\rangle=\{0\}\}$, 
	which is a closed linear subspace of $\Hc^\infty$. 
	Since $\Yc\subseteq \Hc^{-\infty}$ is a $w^*$-closed linear subspace, we obtain 
	a topological linear isomorphism  
	\begin{equation}
		\label{w_irred_proof_eq1}
		\Yc\simeq(\Hc^\infty/\Yc^\perp)'. 
	\end{equation}
	On the other hand, the hypothesis $\pi^{-\infty}(\Sc(G))\Yc\subseteq\Yc$ implies 
	$\pi^\infty(\Sc(G))\Yc^\perp\subseteq\Yc^\perp$ 
	hence, by Lemma~\ref{alg_irred}, we have either $\Yc^\perp=\{0\}$ or $\Yc^\perp=\Hc^\infty$. 
	Now the assertion follows by \eqref{w_irred_proof_eq1}. 
\end{proof}

The next lemma is a more precise form of \cite[Cor.~3.4.2]{Ho77}. 
We give below its short proof. 
For this statement we recall that $\Sc(G)$ is a $*$-subalgebra of the Banach algebra $L^1(G)$, and by $*$-representation of $\Sc(G)$ on a Hilbert space $\Xc$ we mean any $*$-morphism $\sigma\colon \Sc(G)\to\Bc(\Xc)$ without any assumptions on continuity. 

\begin{lemma}\label{isometry}
Let $\pi$ be as above and $\sigma\colon \Sc(G)\to \Bc(\Xc)$ be a $*$-representation on a Hilbert space $\Xc$. 
Assume that the linear map $T\colon \Hc^\infty\to \Xc$ satisfies
$$ T\pi(\varphi) = \sigma (\varphi) T \quad \text{for all } \varphi \in \Sc(G).$$
Then $T$ extends to a bounded operator on $\Hc$ that is  a scalar multiple of an isometry. 
\end{lemma}

\begin{proof}
Let $v_1, v_2\in \Hc^\infty\setminus \{0\}$ be such that 
$\dual{v_1}{v_2}=0$.
The orthogonal projection $P:=\dual{\cdot}{v_1}{v_1}$ belongs to $\Bc(\Hc)^\infty$ by \cite[Ex. 3.1]{BeBe10}. 
By \cite[Thm.~3.4]{Ho77}, there is $\varphi\in \Sc(G)$ with 
$P= \pi(\varphi)$, hence $P=P^*=P^2 = \pi(\psi)$ where $\psi:=\varphi^*\ast \varphi\in\Sc(G)$. 
Then, by hypothesis,  $\sigma(\psi)T=TP=\dual{\cdot}{v_1}Tv_1$, hence
$$ \sigma(\psi) T v_1= Tv_1 \; \; \text{and}\; \; \sigma (\psi) Tv_2=0.$$
It follows that $Tv_1$, respectively $Tv_2$, are eigenvectors for the eigenvalues $1$, respectively $0$, of the operator $\sigma(\psi)\in\Bc(\Xc)$. 
Here $\sigma(\psi)=\sigma(\varphi)^*\sigma(\varphi)$ 
since the mapping $\sigma\colon \Sc(G)\to\Bc(\Xc)$ is a $*$-representation, 
hence $\sigma(\psi)$ is a self-adjoint operator, and then its eigenvectors $Tv_1$ and  $Tv_2$ are orthogonal. 

Therefore we have obtained that for $v_1, v_2\in \Hc^\infty$, 
\begin{equation}\label{isom1}
 \dual{v_1}{v_2}=0 \; \Longrightarrow  \; \dual{Tv_1}{T v_2}=0.
\end{equation}
Let now $v, w\in \Hc^\infty$ be arbitrary with $\Vert v\Vert =\Vert w\Vert=1$. 
Set 
$ v_2 = w- \dual{w}{v} v$.
Then $\dual{v}{v_2}=0$  hence, by \eqref{isom1}, $\dual{Tv}{Tv_2}=0$ , that is, 
 \begin{equation}\label{isom2}
\dual{Tv}{Tw}= \dual{v}{w}\Vert Tv\Vert^2.
\end{equation} 
Interchanging the roles of $w$ and $v$, we obtain  
$\dual{v}{w}\Vert Tv\Vert^2= \dual{v}{w}\Vert Tw\Vert^2$.
It follows that 
$$ \Vert Tv\Vert= \Vert Tw\Vert  \quad\text{if  }\dual{v}{w}\ne 0.$$
If $\dual{v}{w}= 0$ and we put
$w_\epsilon=(w+\epsilon v)/\Vert w+\epsilon v\Vert$ for every $\epsilon>0$, then 
$\dual{v}{w_\epsilon}\ne 0$. Using the above equality we get that 
$$ \Vert Tv\Vert =\Vert w+\epsilon v\Vert^{-1}\Vert T(w+\epsilon v)\Vert =\Vert w+\epsilon v\Vert^{-1}  \Vert Tw+\epsilon Tv\Vert.$$
Letting $\epsilon\to 0$, we get that 
$$ \Vert Tv\Vert= \Vert Tw\Vert  \quad\text{for every} \quad v, w\in \Hc^\infty, \; \Vert v\Vert =\Vert w\Vert =1.$$
Thus $\lambda_0:=\Vert Tv\Vert$ is independent on $v\in\Hc^\infty$ with $\Vert v\Vert=1$. 
Then by \eqref{isom2}
we get 
\begin{equation}\label{isom3}
\dual{T v}{T w }= \lambda_0^2 \dual{v}{w}
 \end{equation} 
for every $v, w\in \Hc^\infty$.
Since $\Hc^\infty$ is dense in $\Hc$, \eqref{isom3} shows that $T$ extends 
as a bounded operator on $\Hc$, $\eqref{isom3}$ holds for $v, w\in \Hc$, hence the extension of $T$ is a scalar multiple of an isometry.  
\end{proof}

\begin{theorem}
\label{unique_nilp}
Let $\pi\colon G\to\Bc(\Hc)$ be an irreducible unitary representation of a 1-connected, nilpotent Lie group $G$. 
Let $\Xc\subseteq\Hc^{-\infty}$ be a $\CC$-linear subspace endowed with the structure of a complex Hilbert space satisfying the hypothesis~\ref{unique_item1} and, additionally,  
\begin{enumerate}
[start=5,label={\rm(H\arabic*)}]
	\item\label{unique_item5}
	 $\pi^{-\infty}(\Sc(G))\Xc\subseteq\Xc$. 
\end{enumerate}
Then either $\Xc=\{0\}$ or $\Xc=\Hc$. 
\end{theorem}

\begin{proof}
We start by the preliminary remark that for any unitary representation $\pi$ of an arbitrary Lie group~$G$, if $0\ne x\in\Hc^{-\infty}$, then there exists $v\in\Hc^\infty$ with $\langle x,v\rangle\ne 0$ hence for a suitable $\varphi\in C_0^\infty(G)$ we have $\langle x,\pi(\varphi)v\rangle\ne 0$, which implies $\pi^{-\infty}(C_0^\infty(G))x\ne \{0\}$. 

We now come back to the proof and assume $\Xc\ne\{0\}$. 
Then, by the above remark along with \eqref{triple_eq2bis} and \ref{unique_item5}, we have 
\begin{equation*}
	\{0\}\ne \pi^{-\infty}(\Sc(G))\Xc\subseteq \Xc\cap\Hc^\infty.
\end{equation*}
On the other hand, the $\CC$-linear space $\pi^{-\infty}(\Sc(G))\Xc$ is invariant to $\pi^{-\infty}(\Sc(G))$ since $\pi^{-\infty}\colon \Sc(G)\to\Lc(\Hc^{-\infty})$ is an algebra representation. 
Since the inclusion $\Hc^\infty\hookrightarrow\Hc^{-\infty}$ intertwines the representations $\pi^\infty$ and $\pi^{-\infty}$ of the algebra $\Sc(G)$, 
an application of Lemma~\ref{alg_irred} then shows that $\pi^{-\infty}(\Sc(G))\Xc=\Hc^\infty$, 
which further implies $\Hc^\infty\subseteq\Xc$. 
Since the representation $\pi$ of the nilpotent Lie group $G$ is irreducible, 
it follows by 
Lemma~\ref{isometry} along with hypothesis \ref{unique_item5} that $\Hc\subseteq \Xc$. 

However, 
as in the proof of Theorem~\ref{unique}, we have the continuous linear operator $\iota\colon\Hc^\infty\to\Xc$ given by 
\begin{equation*}
	(\forall v\in\Hc^\infty, x\in\Xc)\quad 
	\langle x,v\rangle =\langle x,\iota(v)\rangle_\Xc. 
\end{equation*} 
with $\iota(\Hc^\infty)$ dense in $\Xc$. 
Since $\iota(v)=v$ for $v\in \Hc^\infty\cap\Xc$ and we established above that $\Hc\subseteq \Xc$, we actually obtain $\Hc= \Xc$.
\end{proof}

\subsection*{Acknowledgment} 
We wish to thank the Referee for the remarks on our paper and for pointing out \cite[Thm.~2.8]{Ne00}, and indicating how to 
considerably improve the paper, 
by extending of Theorem~\ref{unique} to general Lie groups, 
as well as further developments along the lines of Theorem~\ref{unique_nilp}. 

We also thank Dr. Shubham Bais for kindly providing an electronic copy of \cite{BaMoPi24}.

\end{document}